\documentclass[reqno]{amsart}

\usepackage{mystyle}

\usepackage{etoolbox}

\usepackage[color=cyan!50]{todonotes} 
\presetkeys{todonotes}{inline}{} 
\usepackage{xcolor}

\usepackage[headings]{fullpage}
\usepackage{parskip}

\usepackage[utf8]{inputenc}
\usepackage[T1]{fontenc}
\usepackage[british]{babel}
\usepackage[pdfencoding=auto]{hyperref}

\usepackage{microtype}
\usepackage{booktabs}

\usepackage{amsfonts,amssymb,bbm,mathrsfs,stmaryrd}
\usepackage{amsmath}

\usepackage{mathtools}
\usepackage[noabbrev]{cleveref}
\usepackage{centernot}
\usepackage[shortlabels]{enumitem}
\usepackage{url}
\usepackage{tikz}
\usepackage{pict2e}
\usepackage{framed}

\usetikzlibrary{matrix}
\usetikzlibrary{positioning}
\usetikzlibrary{arrows}
\tikzset{> =stealth}
\usetikzlibrary{arrows.meta}
\tikzset{normalHead/.tip={Triangle[open,angle=60:4pt]},}
\tikzset{normalTail/.tip={Triangle[reversed,open,angle=60:4pt]},}

\theoremstyle{plain}
\newtheorem{theorem}{Theorem}[section]
\newtheorem{lemma}[theorem]{Lemma}
\newtheorem{proposition}[theorem]{Proposition}

\theoremstyle{definition}
\newtheorem{definition}[theorem]{Definition}

\newtheorem{example}{Example}
\AtBeginEnvironment{example}{\vspace{11pt}\begin{leftbar}\vspace{-11pt}}
\AtEndEnvironment{example}{\end{leftbar}}

\AtBeginEnvironment{construction}{\begin{leftbar}}
\AtEndEnvironment{construction}{\end{leftbar}}

\theoremstyle{remark}
\newtheorem{remark}{Remark}

\renewcommand{\epsilon}{\varepsilon}
\renewcommand{\phi}{\varphi}

\newcommand{\inv}{^{-1}}

\newcommand{\SWSExt}{\mathrm{CExt}}
\newcommand{\Eq}{\mathrm{Eq}}

\DeclareMathOperator{\Aut}{Aut}

\newcommand{\splitext}[6]{
\tikz[baseline]{
\newdimen{\mylabelwidth}
\newdimen{\skipwidth}
\node[anchor=base] (A) {\hspace*{\dimexpr0.5pt-\pgfkeysvalueof{/pgf/inner xsep}}${#1}$};
\settowidth{\mylabelwidth}{\pgfinterruptpicture {$#2$} \endpgfinterruptpicture}
\pgfmathsetlength{\skipwidth}{max(\mylabelwidth,10pt)}
\node[right] (B) at ([xshift=\skipwidth+12pt]A.east) {${#3}$};
\settowidth{\mylabelwidth}{\pgfinterruptpicture {$#4$} \endpgfinterruptpicture}
\settowidth{\skipwidth}{\pgfinterruptpicture {$#5$} \endpgfinterruptpicture}
\pgfmathsetlength{\skipwidth}{max(\skipwidth,\mylabelwidth,10pt)}
\node[right] (C) at ([xshift=\skipwidth+12pt]B.east) {${#6}$\hspace*{\dimexpr0.5pt-\pgfkeysvalueof{/pgf/inner xsep}}};
\draw[normalTail->] (A) to node [above] {${#2}$} (B);
\draw[transform canvas={yshift=0.5ex},-normalHead] (B) to node [above] {${#4}$} (C);
\draw[transform canvas={yshift=-0.5ex},->] (C) to node [below] {${#5}$} (B);
}}

\newcommand{\normalext}[5]{
\tikz[baseline]{
\newdimen{\mylabelwidth}
\newdimen{\skipwidth}
\node[anchor=base] (A) {\hspace*{\dimexpr0.5pt-\pgfkeysvalueof{/pgf/inner xsep}}${#1}$};
\settowidth{\mylabelwidth}{\pgfinterruptpicture {$#2$} \endpgfinterruptpicture}
\pgfmathsetlength{\skipwidth}{max(\mylabelwidth,12pt)}
\node[right] (B) at ([xshift=\skipwidth+10pt]A.east) {${#3}$};
\settowidth{\mylabelwidth}{\pgfinterruptpicture {$#4$} \endpgfinterruptpicture}
\pgfmathsetlength{\skipwidth}{max(\mylabelwidth,10pt)}
\node[right] (C) at ([xshift=\skipwidth+10pt]B.east) {${#5}$\hspace*{\dimexpr0.5pt-\pgfkeysvalueof{/pgf/inner xsep}}};
\draw[normalTail->] (A) to node [above] {${#2}$} (B);
\draw[-normalHead] (B) to node [above] {${#4}$} (C);
}}

\title{Baer sums for a natural class of monoid extensions}
\author[P. F. Faul]{Peter F. Faul}
\address{Department of Pure Mathematics and Statistical Sciences\\ University of Cambridge}
\email{peter@faul.io}

\date{\today}


\begin{document}

\maketitle

\begin{abstract}
    It is well known that the set of isomorphism classes of extensions of groups with abelian kernel is characterized by the second cohomology group. In this paper we generalise this characterization of extensions to a natural class of extensions of monoids, the cosetal extensions. An extension $\normalext{N}{k}{G}{e}{H}$ is cosetal if for all $g,g' \in G$ in which $e(g) = e(g')$, there exists a (not necessarily unique) $n \in N$ such that $g = k(n)g'$. These extensions generalise the notion of special Schreier extensions, which are themselves examples of Schreier extensions. Just as in the group case where a semidirect product could be associated to each extension with abelian kernel, we show that to each cosetal extension (with abelian group) kernel, we can uniquely associate a weakly Schreier split extension. The characterization of weakly Schreier split extensions is combined with a suitable notion of a factor set to provide a cohomology group granting a full characterization of cosetal extensions, as well as supplying a Baer sum. 
\end{abstract}

\section{Introduction}\label{sec:Introduction}
\subsection*{Group cohomology}
The second cohomology group corresponding to group extensions with abelian kernels is a classical piece of mathematics. We associate to each such extension $\normalext{N}{k}{G}{e}{H}$ an action $\phi$ of $H$ on $N$. We do so by noting that, since $N$ is normal, it is closed under conjugation by $G$. This conjugation gives an action $\alpha\colon G \to \Aut(N)$ and since $N$ is abelian, $\alpha k$ is the zero morphism. As $e$ is the cokernel of $k$, we then get that $\alpha$ uniquely extends to a map $\phi \colon H \to \Aut(N)$ --- the desired action of $H$ on $N$.

We can then collect all isomorphism classes of extensions with the same action together in a set $\mathrm{Opext}(H,N,\phi)$ and show that this set is isomorphic in a natural way to the abelian group of factor sets quotiented by inner factor sets. This allows $\mathrm{Opext}(H,N,\phi)$ to inherit an abelian group structure called the Baer sum. For more on this, see \cite{maclane2012homology}.

\subsection*{Monoid cohomology}
Generalising this to the setting of extensions of monoids presents some difficulties. Notably, in the above we made crucial use of conjugation, which is not something available in the monoid setting.

Much work has been done to get around this problem. In \cite{redei1952verallgemeinerung}, Schreier extensions of monoids were introduced. An extension $\normalext{N}{k}{G}{e}{H}$ is Schreier if in each fibre $e^{-1}(h)$ there exists an element $u_h$ such that for all $g \in e^{-1}(h)$ there exists a unique $n \in N$ such that $g = k(n)u_h$. This means that the fibre $e^{-1}(h)$ is equal to the coset $Nu_h$.

Although closer to the structure of a group extension, this setting is not quite enough to adapt our original argument and extract an action. However, if an action is supplied --- that is, if Schreier extensions of a monoid $H$ by an $H$-module $N$ are considered --- then such extensions are classified by a cohomology group (as seen in \cite{tuen1976nonabelianextensions}). This is further generalised to cohomology groups for extensions of $H$ by $H$-semimodules in \cite{patchkoria1979schreier, patchkoria1977extensions}.

In \cite{martins2016baer}, a class of extensions are considered which have enough in common with the group setting that an action can be extracted from the extension itself. The idea behind these \emph{special Schreier} extensions is as follows.

An extension $\normalext{N}{k}{G}{e}{H}$ is special Schreier when the kernel equivalence split extension of $e$ is a Schreier split extension. Translating this into familiar terms, an extension is special Schreier if and only if for each $e(g) = e(g')$ there exists a unique element $n \in N$ such that $k(n)g' = g$. It is clear that special Schreier extensions are Schreier, but that the converse is not in general true.

To extract the action we observe that $e(g) = e(gk(n))$ and apply the special Schreier property, which says that there is a unique element $\alpha(g,n)$ such that $k\alpha(g,n) \cdot g = g \cdot k(n)$. Notice that if we were in the group setting we would have that $\alpha(g,n) = g \cdot k(n) \cdot g\inv$ and so this action generalises the one from the group case. This action then extends as before to one of $H$ on $N$.

The authors of \cite{martins2016baer} then consider isomorphism classes $\mathrm{SExt}(H,N,\phi)$ of extensions associated to the action $\phi$ and are able to classify these extensions using a cohomology group corresponding to a generalised notion of factor sets, and thus imbue $\mathrm{SExt}(H,N,\phi)$ with a Baer sum.

In \cite{faul2019characterization}, weakly Schreier split extensions, a generalization of Schreier split extensions, were characterized in a way that suggested the possibility of defining a cohomology derived from the analogous special weakly Schreier extensions. We will demonstrate that this approach succeeds and a coholomogy group can be associated to this class of extensions.

\subsection*{Outline}
In this paper we generalise the notion of a special Schreier extension, doing away with the uniqueness requirements. We call these extensions \emph{cosetal} because of their relation to cosets. Cosetal extensions are shown to be in one to one correspondance with extensions whose associated kernel equivalence split extension is weakly Schreier.

It is shown that analogous to the characterization of weakly Schreier split extensions in terms of an admissible quotient and a compatible action, such data can be uniquely associated to a cosetal extension.

We then consider isomorphism classes of extensions with the same associated data and characterize them using a cohomology group defined in terms of a natural weakening of factor sets in our setting. This naturally yields a Baer sum.
%

\section{Background}\label{sec:Background}
This paper makes extensive use of the characterization of \emph{weakly Schreier split extensions} in terms of admissible equivalence relations and compatible actions. Here we outline the basic results and motivation found in \cite{faul2019characterization}.

\begin{definition}
    A split extension $\splitext{N}{k}{G}{e}{s}{H}$ is \emph{weakly Schreier} if for each $g \in G$ there exists a (not necessarily unique) $n \in N$ such that $g = k(n)se(g)$.
\end{definition}

This generalises the notion of a \emph{Schreier split extension} which requires that for each $g$ there is a unique $n \in N$ such that $g = k(n)se(g)$.

Natural examples of weakly Schreier extensions are the Artin glueings of frames \cite{wraith1974glueing, faul2019artin} and Billhardt's \cite{billhardt1992wreath} $\lambda$-semidirect products of inverse monoids \cite{faul2020lambda}.

Given a weakly Schreier split extension $\splitext{N}{k}{G}{e}{s}{H}$, it is clear that the set map $f\colon N \times H \to G$ with $f(n,h) = k(n)s(h)$ is surjective. We can thus quotient $N \times H$ by the equivalence relation defined by $(n,h) \sim (n',h')$ if and only if $k(n)s(h) = k(n')s(h')$.

This equivalence relation will always satisfy the following four properties.
\begin{enumerate}
    \item $(n,1) \sim (n',1)$ implies $n = n'$,
    \item $(n,h) \sim (n',h')$ implies $h = h'$,
    \item $(n,h) \sim (n',h)$ implies that $(xn,h) \sim (xn',h)$ for all $x \in N$ and
    \item $(n,h) \sim (n',h)$ implies that $(n,hx) \sim (n',hx)$ for all $x \in H$.
\end{enumerate}
Any equivalence relation $E$ on $N \times H$ satisfying the above properties is called \emph{admissible}.

Similarly, given a weakly Schreier split extension $\splitext{N}{k}{G}{e}{s}{H}$, there exist maps $q\colon G \to N$ in which $g = kq(g)se(g)$ for all $g \in G$. Given such a map we can define a function $\alpha\colon H \times N \to N$ where $\alpha(h,n) = q(s(h)k(n))$.

This function $\alpha$ behaves like an action with respect to the associated admissible equivalence relation in the following way.
\begin{enumerate}
    \item $(n_1,h) \sim (n_2,h)$ implies $[n_1\alpha(h,n),h] = [n_2\alpha(h,n),h]$ for all $n \in N$,
    \item $(n,h') \sim (n',h')$ implies $[\alpha(h,n),hh'] = [\alpha(h,n'),hh']$ for all $h \in H$,
    \item $[\alpha(h,nn'),h] = [\alpha(h,n)\cdot\alpha(h,n'),h]$,
    \item $[\alpha(hh',n),hh'] = [\alpha(h,\alpha(h',n)),hh']$,
    \item $[\alpha(h,1),h] = [1,h]$,
    \item $[\alpha(1,n),1] = [n,1]$.
\end{enumerate}
Notice in particular the last four conditions which, in the first component, are just the usual identities satisfied by an action.

Any function satisfying the above identities with respect to an admissible equivalence relation $E$, we call a \emph{compatible action}.

Now if we assume that we have extracted an admissible equivalence relation $E$ and a compatible action $\alpha$ from a weakly Schreier split extension $\splitext{N}{k}{G}{e}{s}{H}$, we can equip the quotient $(N \times H)/E$ with a multiplication given by
\[
    [n,h][n',h'] = [n\alpha(h,n),hh'].
\]
The monoid $(N \times H)/E$ is isomorphic to $G$, where $[n,h]$ is sent to $k(n)s(h)$. In fact, we can construct a weakly Schreier extension $\splitext{N}{k'}{(N \times H)/E}{e'}{s'}{H}$ in which $k'(n) = [n,1]$, $e'([n,h]) = h$ and $s'(h) = [1,h]$. The isomorphism mentioned above is then an isomorphism of split extensions.

Furthermore, if we chose a different map $q \colon G \to N$ originally, the resulting compatible action would still give the same multiplication. This suggests that we identify compatible actions which give the same multiplication. This occurs precisely when $(\alpha(h,n),h) \sim (\alpha'(h,n),h)$ for all $n \in N$ and $h \in H$.

We can do this process in the other direction too. Starting with an admissible equivalence relation $E$ and a class of compatible actions $[\alpha]$, we can construct the associated weakly Schreier extension $\splitext{N}{k}{(N \times H)/E}{e}{s}{H}$ and from this extract the associated admissible equivalence relation $E'$ and class of compatible actions $[\alpha']$. Indeed, we find that $E = E'$ and $[\alpha] = [\alpha']$.

Thus, we have that weakly Schreier extensions are precisely characterized by admissible equivalence relations and compatible actions.

The final point worth emphasising is that the split short five lemma fails for weakly Schreier split extensions. Thus, there are morphisms of split extensions between weakly Schreier extensions which are not necessarily isomorphisms. It is the case however, that all such morphisms are unique and so the category of weakly Schreier extensions is a preorder. This then imbues the set of pairs $(E,[\alpha])$ of admissible equivalence relations and compatible actions with an order structure, where $(E,[\alpha]) \le (E',[\alpha'])$ if and only if $(\alpha(h,n),h) \sim_{E'} (\alpha'(h,n),h)$ and $(n,h) \sim_E (n',h)$ implies $(n,h) \sim_{E'} (n',h)$.

\section{Cosetal Extensions}
In this paper we consider a class of extensions we call cosetal extensions, which have much in common with extensions of groups, specifically pertaining to their relationship with cosets of the kernel.

\begin{definition}
    An extension $\normalext{N}{k}{G}{e}{H}$ is \emph{cosetal} if for all $g, g' \in G$ in which $e(g) = e(g')$, there exists an $n \in N$ such that $k(n)g' = g$.
\end{definition}

\begin{proposition}
    An extension $\normalext{N}{k}{G}{e}{H}$ is cosetal if and only if $Ng = Ng'$ whenever $e(g) = e(g')$. Furthermore in this case the monoid of cosets is isomorphic to $H$.
\end{proposition}

\begin{proof}
    Suppose the extension $\normalext{N}{k}{G}{e}{H}$ is cosetal.
    
    Suppose $e(g) = e(g')$ and consider $x \in Ng$. Notice that $e(x) = e(g) = e(g')$ thus there exists an $n \in N$ such that $x = k(n)g'$. Thus $x \in Ng'$  and so $Ng \subseteq Ng'$. By a symmetric argument we get that $Ng' \subseteq Ng$, which gives the desired result.
    
    Let $\normalext{N}{k}{G}{e}{H}$ be an extension and suppose $Ng = Ng'$ whenever $e(g) = e(g')$.
    
    This means that $g \in Ng'$ which in turn means that there exists an $n \in N$ such that $g = k(n)g'$, giving us that the extension is cosetal.
    
    If $G/N$ is the monoid of cosets then the map sending $Ng$ to $e(g)$ can easily be seen to be an isomorphism.
\end{proof}

%
%
%
%
\begin{remark}
This seems to be a very natural concept and so I would not be surprised if it has already been defined in the literature already. I would be interested to know if this is the case.
\end{remark}

The following lemma follows immediately from the definition.

\begin{lemma}\label{cor:coset}
    Let $\normalext{N}{k}{G}{e}{H}$ be cosetal and let $s$ and $s'$ be (set-theoretic) sections of $e$. Then there exists a function $t\colon H \to N$ such that $s(h) = kt(h) \cdot s'(h)$ for all $h \in H$.
\end{lemma}

There is a connection between cosetal extensions and weakly Schreier extensions of monoids involving the kernel equivalence.

If $\normalext{N}{k}{G}{e}{H}$ is an extension, then the \emph{kernel equivalence split extension} of $e$ is the diagram
\[
    \splitext{N}{(k,0)}{\Eq(e)}{\pi_2}{(1_G,1_G)}{G}
\]
where $\Eq(e)$ is the monoid of all pairs $(g,g')$ in which $e(g) = e(g')$, $(k,0)(n) = (k(n),1)$, $\pi_2(g,g') = g'$ and $(1_G,1_G)(g) = (g,g)$.

\begin{proposition}
    An extension $\normalext{N}{k}{G}{e}{H}$ is cosetal if and only if the associated kernel equivalence split extension is weakly Schreier.
\end{proposition}

\begin{proof}
    Let $\normalext{N}{k}{G}{e}{H}$ be an extension and consider the kernel equivalence split extension \[\splitext{N}{(k,0)}{\mathrm{Eq}(e)}{\pi_2}{(1_G,1_G)}{G}.\]
    
    For it to be weakly Schreier we require that for all $(g,g') \in \Eq(e)$ there exists an $n \in N$ such that $(g,g') = (k,0)(n) \cdot (1_G,1_G)\pi_2(g,g') = (k(n)g',g')$. Thus, we see that this property will hold for all pairs if and only if whenever $e(g) = e(g')$ there exists an $n \in N$ such that $k(n)g' = g$, which is precisely the cosetal condition.
\end{proof}

%

\subsection*{The link to special Schreier extensions}

In \cite{martins2016baer}, a Baer sum was determined for the class of special Schreier extensions with abelian kernel. Special Schreier extensions are those whose associated kernel equivalence split extension is a Schreier split extension. Since all Schreier split extensions are weakly Schreier split extensions, it is clear that all special Schreier extensions are cosetal. We should verify that there are cosetal extensions which are not special Schreier.
%
%


\begin{proposition}
    If $\splitext{N}{k}{G}{e}{s}{H}$ is a weakly Schreier split extension with $N$ a group, then it is cosetal.
\end{proposition}

\begin{proof}
    We must show that if $e(a) = e(b)$ that there exists an $n \in N$ such that $k(n)b = a$. Since our extension is weakly Schreier there exists $n_a$ and $n_b$ such that $a = k(n_a)se(a)$ and $b = k(n_b)se(b)$. Since $se(a) = se(b)$ we can write $b = k(n_b)se(a)$. Notice then that $k(n_an_b\inv)b = a$ and so we are done.
\end{proof}

This strongly suggests that there are cosetal extensions which are not special Schreier. In order to make this concrete, we demonstrate an example of a weakly Schreier extension with a group kernel, which is not Schreier.

\begin{example}
    We make use of the course quotient construction in \cite{faul2019characterization}. Let us take the integers $\mathbb{Z}$ with addition as the group kernel and the two element meet-semilattice $H = \{1,h\}$ as the cokernel.
    
    We then generate the coarse quotient on $\mathbb{Z} \times H$ which has that $(n,1) \sim (n',1)$ if and only if $n = n'$ and $(n,h) \sim (n',h)$ for all $n$ and $n' \in \mathbb{Z}$. Multiplication is given by
    \begin{enumerate}
        \item $[n,1][n',1] = [nn',1]$ and
        \item $x[n,h] = [n,h] = [n,h]x$ for all $x$ in the quotient.
    \end{enumerate}
    This reduces to $\mathbb{Z}\sqcup\{\infty\}$ where $x + \infty = \infty = \infty + x$ for $x \in \mathbb{Z}\sqcup\{\infty\}$.
    Now we can consider the extension $\splitext{\mathbb{Z}}{k}{\mathbb{Z}\sqcup\{\infty\}}{e}{s}{H}$ in which $k$ is the inclusion, $e(n) = 1$ for all $n \in \mathbb{Z}$, $e(\infty) = h$ and $s(1) = 1$ and $s(h) = \infty$.
    
    Now, the kernel equivalence split extension will not be a Schreier split extension as there will be many $n \in \mathbb{Z}$ for which $(\infty,\infty) = (k(n),0) \cdot (\infty, \infty)$.
\end{example}

\section{The cosetal extension problem}

\subsection*{Extending the admissible equivalence relation and compatible action}

Since we are interested in generalising the work done on group extensions to this new setting, we shall henceforth assume that the kernel $N$ is always an abelian group.

Despite a cosetal extension $\normalext{N}{k}{G}{e}{H}$ not in general being a split extension, there is a version of the weakly Schreier condition that holds for all set theoretic splittings of $e$. For convenience we assume that all set theoretic sections $s$ of $e$ which we consider, preserve the identity.

\begin{proposition}\label{prop:semibiproducts}
    Let $\normalext{N}{k}{G}{e}{H}$ be cosetal and let $s$ be a section of $e$. Then for all $g \in G$ there exists an $n \in N$, such that $g = k(n)se(g)$.
\end{proposition}

\begin{proof}
    Simply observe that $e(g) = ese(g)$ and apply the cosetal property to $g$ and $se(g)$.
\end{proof}

In \cite{martins2020semi}, a class of extensions more general than weakly Schreier extensions, called \emph{semi-biproducts}, are considered. These extensions $\normalext{N}{k}{G}{e}{H}$ have as additional data a set theoretic section $s$ of $e$ and also a set theoretic retraction $q$ of $k$. Together they satisfy the weakly Schreier condition that for all $g \in G$, $g = kq(g)se(g)$. It is clear from \cref{prop:semibiproducts}, that cosetal extensions can be equipped with $q$ and $s$ turning them into semi-biproducts.

It was shown (albeit in a different, but equivalent form) that the characterization of weakly Schreier extensions in \cite{faul2019characterization} generalises naturally to semi-biproducts. When $N$ is an abelian group and $\normalext{N}{k}{G}{e}{H}$ is assumed to be cosetal, we obtain a characterization even more closely resembles the weakly Schreier characterization.

\begin{proposition}
    Let $\normalext{N}{k}{G}{e}{H}$ be a cosetal extension and let $s$ be a section of $e$. The equivalence relation $E_s$, defined by $(n,h) \sim (n',h')$ if and only if $k(n)s(h) = k(n')s(h')$, is admissible.
\end{proposition}

\begin{proof}
    Notice that if $(n,1) \sim (n',1)$, then $k(n) = k(n')$, since $s$ preserves the unit. This implies that $n = n'$ as required.
    
    Now if $(n,h) \sim (n',h')$, then $k(n)s(h) = k(n')s(h')$. Applying $e$ to both sides yields $h = h'$ as required.
    
    If $k(n)s(h) = k(n')s(h)$ then of course $k(x)k(n)s(h) = k(x)k(n')s(h)$. Since $k$ is a monoid homomorphism, this gives that $(n,h) \sim (n',h)$ implies that $(xn,h) \sim (xn',h)$ for all $x \in N$.
    
    Finally, suppose that $k(n)s(h) = k(n')s(h)$ and consider $k(n)s(hx)$ and $k(n')s(hx)$. Notice that $e(s(h)s(x)) = es(hx)$ and so, since our extension is cosetal, we have that there exists an $a \in N$ such that $k(a)s(hx) = s(h)s(x)$. Now consider the following calculation.
    \begin{align*}
        k(a)k(n)s(hx)   &= k(n)k(a)s(hx) \\
                        &= k(n)s(h)s(x) \\
                        &= k(n')s(h)s(x) \\
                        &= k(a)k(n')s(hx).
    \end{align*}
    Here the first equality holds because $N$ is an abelian group. Now since $a$ is invertible it follows that $k(n)s(hx) = k(n')s(hx)$.
    This shows that for all $x \in H$, $(n,h) \sim (n',h)$ implies $(n,hx) \sim (n',hx)$, and hence $E$ is admissible.
\end{proof}

The above result required an arbitrary choice of splitting. The following proposition demonstrates that the choice of splitting does not matter.

\begin{proposition}
    Let $\normalext{N}{k}{G}{e}{H}$ be a cosetal extension and let $s$ and $s'$ be sections of $e$. Then the associated equivalence relations $E_s$ and $E_{s'}$ are equal.
\end{proposition}

\begin{proof}
    Without loss of generality, it is sufficient to show that $E_{s} \subseteq E_{s'}$. By \cref{cor:coset} there exists a function $t\colon H \to N$ such that $kt(h)s(h) = s'(h)$.
    
    Suppose that $(n,h) \sim_s (n',h)$. This means that $k(n)s(h) = k(n')s(h)$. We now have
    \begin{align*}
        k(n)s'(h)   &= k(n)kt(h)s(h) \\
                    &= kt(h)k(n)s(h) \\
                    &= kt(h)k(n')s(h) \\
                    &= k(n')kt(h)s(h) \\
                    &= k(n')s'(h).
    \end{align*}
    Hence $(n,h) \sim_{s'} (n',h)$ as required.
\end{proof}

For admissible equivalence relations, it makes sense to consider the following two operations.
\begin{enumerate}
    \item $n' \ast [n,h]  = [n'n,h]$ and
    \item $[n,h] \ast h' = [n,hh']$.
\end{enumerate}

We also find that each cosetal extension $\normalext{N}{k}{G}{e}{H}$ has a unique equivalence class of actions compatible with the admissible equivalence relation. The idea is to consider the kernel equivalence split extension $\splitext{N}{(k,0)}{\Eq(e)}{\pi_2}{(1_G,1_G)}{G}$ which we know to be weakly Schreier and to take one of the compatible actions $\alpha\colon G \times N \to N$ associated to it. Then we simply define the 'action' $\phi \colon H \times N \to N$ as $\alpha(s \times 1_N)$ for some section $s$. Before we can show this action is compatible, we prove the following useful lemma.

\begin{lemma}\label{lem:conjugation}
    Let $\normalext{N}{k}{G}{e}{H}$ be cosetal and let $\splitext{N}{(k,0)}{\Eq(e)}{\pi_2}{(1_G,1_G)}{G}$ be its associated weakly Schreier kernel equivalence split extension. Then if $\alpha\colon G \times N \to N$ is a compatible action, we have that $k\alpha(g,n)g = gk(n)$.
\end{lemma}

\begin{proof}
    Recall that all compatible actions $\alpha$ come from particular Schreier retractions. Let $q$ be a Schreier retraction for $\splitext{N}{(k,0)}{\Eq(e)}{\pi_2}{(1_G,1_G)}{G}$ and define
    \begin{align*}
        \alpha(g,n) &= q((1_G,1_G)(g) \cdot (k,0)(n)) \\
                    &= q(gk(n),g).
    \end{align*}
    Notice that we have
    \begin{align*}
        (gk(n),g)   &= (k,0)q(gk(n),g) \cdot (1_G,1_G)\pi_2(gk(n),g) \\
                    &= (k,0)\alpha(g,n) \cdot (1_G,1_G)\pi_2(gk(n),g) \\
                    &= (k\alpha(g,n),1) \cdot (g,g) \\
                    &= (k\alpha(g,n)g,g).
    \end{align*}
    Thus we can deduce that $k\alpha(g,n)g = gk(n)$ as required.
\end{proof}

\begin{proposition}
    Let $\normalext{N}{k}{G}{e}{H}$ be cosetal, let $s$ be a section of $e$ and let $\alpha\colon G \times N \to N$ be a compatible action associated to its (weakly Schreier) kernel equivalence split extension. Then the map $\phi = \alpha(s \times 1_N)$ is compatible with the associated admissible equivalence relation $E$.
\end{proposition}

\begin{proof}
    We begin by showing that $(n,h) \sim (n',h)$ implies that $(n\phi(h,x),h) \sim (n'\phi(h,x),h)$ for all $x \in N$.
    
    Consider $k(n)k\phi(h,x)s(h)$. Using \cref{lem:conjugation} and the fact that $\phi(h,x) = \alpha(s(h),x)$ we get
    \begin{align*}
        k(n)k\phi(h,x)s(h)      &= k(n)s(h)k(x) \\
                                &= k(n')s(h)k(x) \\
                                &= k(n')k\phi(h,x).
    \end{align*}
    This gives the desired result.
    
    Now let us show that $(n,h) \sim (n',h)$ implies that $(\phi(x,n),xh) \sim (\phi(x,n'),xh)$.
    
    Let $a \in N$ be such that $k(a)s(xh) = s(x)s(h)$ and consider
    \begin{align*}
        k(a)k\phi(x,n)s(xh) &= k\phi(x,n)s(x)s(h) \\
                            &= s(x)k(n)s(h) \\
                            &= s(x)k(n')s(h) \\
                            &= k(a)k\phi(x,n')s(xh).
    \end{align*}
    Again, since $a$ is invertible we get that $k\phi(x,n)s(xh) = k\phi(x,n')s(xh)$ as required.
    
    Next we show that $(\phi(h,nn'),h) \sim (\phi(h,n)\phi(h,n'),h)$.
    
    Observe the following calculation.
    \begin{align*}
        k\phi(h,nn')s(h)    &= s(h)k(n)k(n') \\
                            &= k\phi(h,n)s(h)k(n') \\
                            &= k\phi(h,n)k\phi(h,n')s(h).
    \end{align*}
    This gives the desired result.
    
    Next we show that $(\phi(hh',n),hh') \sim (\phi(h,\phi(h',n)))$.
    
    Let $a \in N$ be such that $k(a)s(hh') = s(h)s(h')$ and consider the following.
    \begin{align*}
        k(a)k\phi(hh',n)s(hh')  &= k(a)s(hh')k(n) \\
                                &= s(h)s(h')k(n) \\
                                &= s(h)k\phi(h',n)s(h') \\
                                &= k\phi(h,\phi(h',n))s(h)s(h') \\
                                &= k(a)k\phi(h,\phi(h',n))s(hh').
    \end{align*}
    This gives that $k\phi(hh',n)s(hh') = k\phi(h,\phi(h',n))s(hh')$, which in turn yields our desired result.
    
    Finally, we must show that $(\phi(h,1),h) \sim (1,h)$ and that $(\phi(1,n),1) \sim (n,1)$.
    
    For the first observe that $k\phi(h,1)s(h) = s(h)k(1) = s(h)$ and for the second that $k\phi(1,n)s(1) = k(n)$. Notice that the latter case in fact implies that $\phi(1,n) = n$.
    
    Thus, we have shown that each of the six necessary conditions are satisfied and so $\phi$ is compatible with $E$.
\end{proof}

Our construction of $\phi$ required an arbitrary choice of $\alpha$. We now show this choice does not matter.

\begin{proposition}
    Let $\normalext{N}{k}{G}{e}{H}$ be cosetal, let $s$ be a section of $e$ and let $\alpha\colon G \times N \to N$ and $\alpha'\colon G \times N \to N$ be compatible actions associated to its kernel equivalence split extension. Then the maps $\phi = \alpha(s \times 1_N)$ and $\phi' = \alpha'(s \times 1_N)$ are equivalent compatible actions with respect to the admissible equivalence relation $E$.
\end{proposition}

\begin{proof}
    We must show that $(\phi(h,n),h) \sim (\phi'(h,n),h)$ for all $n \in N$ and $h \in H$. This follows immediately from \cref{lem:conjugation} applied to $\alpha$ and $\alpha'$ as $k\phi(h,n)s(h) = s(h)k(n) = k\phi'(h,n)s(h)$.
\end{proof}

In fact, the choice of splitting does not matter either.

\begin{proposition}
    Let $\normalext{N}{k}{G}{e}{H}$ be cosetal, let $s$ and $s'$ be sections of $e$ and let $\alpha\colon G \times N \to N$ be a compatible action associated to its kernel equivalence split extension. Then the maps $\phi = \alpha(s \times 1_N)$ and $\phi' = \alpha(s' \times 1_N)$ are equivalent with respect to the associated admissible equivalence relation $E$.
\end{proposition}

\begin{proof}
    We must show that $(\phi(h,n),h) \sim (\phi'(h,n),h)$. By \cref{cor:coset}, we have a function $t\colon H \to N$ such that $kt(h)s'(h) = s(h)$. Now consider
    \begin{align*}
        k\phi'(h,n)s(h) &= k\phi'(h,n)kt(h)s'(h) \\
                        &= kt(h)k\phi'(h,n)s'(h) \\
                        &= kt(h)s'(h)k(n) \\
                        &= s(h)k(n) \\
                        &= k\phi(h,n)s(h).
    \end{align*}
    This completes the proof.
\end{proof}

So given a cosetal extension $\normalext{N}{k}{G}{e}{H}$, we can associate a unique admissible equivalence relation $E$ and a unique equivalence class of compatible actions $[\phi]$.

\subsection*{Factor sets and the Baer sum} 

We can now partition the set of isomorphism classes of cosetal extensions, parameterised by an admissible equivalence relation and a compatible action.

\begin{definition}
    Let $\SWSExt(H,N,E,[\phi])$ be the set of isomorphism classes of cosetal extension \[\normalext{N}{k}{G}{e}{H},\] such that $E$ is the associated admissible equivalence relation and $[\phi]$ the associated class of compatible actions.
\end{definition}


As in the case of extensions groups or special Schreier extensions of monoids, the extensions in $\SWSExt(H,N,E,\phi)$ correspond to some notion of factor sets.

Let $\normalext{N}{k}{G}{e}{H}$ be a cosetal extension and let $s$ be a section of $e$. Recall that $e(s(h)s(h')) = hh' = e(s(hh'))$ and so there exists an $x \in N$ such that $xs(hh') = s(h)s(h')$. Let $g\colon H \times H \to N$ be a function such that $g(h,h')s(hh') = s(h)s(h')$. Notice that we may always choose $g$ such that $g(x,1) = 1 = g(1,x)$.

\begin{definition}
    Let $\normalext{N}{k}{G}{e}{H}$ be a cosetal extension and let $s$ be a section of $e$. Then an \emph{associated factor set} is function $g_s\colon H \times H \to N$ for which $g_s(x,1) = 1 = g_s(1,x)$ and $g_s(h,h')s(hh') = s(h)s(h')$ for all $h,h' \in H$.
\end{definition}

The following result will motivate our definition of a general factor set below.

\begin{proposition}
    Let $\normalext{N}{k}{G}{e}{H}$ be a cosetal extension, $s$ be a section of $e$, $g_s$ an associated factor set and $E$ and $[\phi]$ the associated admissible equivalence relation and class of compatible actions respectively. Then
    \[
        (g(x,y)g(xy,z),xyz) \sim (\phi(x,g(y,z))g(x,yz),xyz).
    \]
\end{proposition}

\begin{proof}
    We must check that $kg(x,y)kg(xy,z)s(xyz) = k\phi(x,g(y,z))kg(x,yz)s(xyz)$.
    
    The left hand side gives
    \begin{align*}
        kg(x,y)kg(xy,z)s(xyz)   &= kg(x,y)s(xy)s(z) \\
                                &= s(x)s(y)s(z).
    \end{align*}
    The right side similarly gives
    \begin{align*}
        k\phi(x,g(y,z))kg(x,yz)s(xyz)   &= k\phi(x,g(y,z))s(x)s(yz) \\
                                        &= s(x)kg(y,z)s(yz) \\
                                        &= s(x)s(y)s(z).
    \end{align*}
    Thus it follows that these two pairs are equivalent.
\end{proof}

\begin{definition}
    A map $g\colon H \times H \to N$ is a \emph{factor set} with respect to an admissible equivalence relation $E$ and a compatible action $\phi$ if $g(x,1) = 1 = g(1,x)$ and
    \[
        (g(x,y)g(xy,z),xyz) \sim (\phi(x,g(y,z))g(x,yz),xyz).
    \]
\end{definition}

Notice that the first components of the equivalence are just the usual factor set definition for special Schreier extensions.

Given an abelian group $N$ and a monoid $H$ with the additional data of an admissible equivalence relation $E$ on $N \times H$, a compatible action $\phi$ and a factor set $g$, we can construct an extension.

\begin{lemma}\label{lem:calc}
    Let $E$ be an admissible equivalence relation on $N \times H$ with $N$ an abelian group. Then if $[n,h] = [n',h]$, we have $[xny,hz] = [xn'y,hz]$ for all $x,y \in N$ and $z \in H$.
\end{lemma}

\begin{proof}
    Suppose $[n,h] = [n'h]$. Then consider
    \begin{align*}
        [xny,hz]    &= xy \ast [n,h] \ast z \\
                    &= xy \ast [n',h] \ast z \\
                    &= [xn'y,hz].
    \end{align*}
    This completes the proof.
\end{proof}

\begin{proposition}
    Let $N$ be an abelian group, $H$ a monoid, $E$ an admissible equivalence relation, $\phi$ a compatible action and $g$ a factor set. Then $(N \times H)/E$ can be equipped with a multiplication
    \[
        [n,h][n',h'] = [n\phi(h,n')g(h,h'),hh'],
    \]
    which makes it into a monoid with identity $[1,1]$. We call this monoid $(N\times H)/E^\phi_g$.
\end{proposition}

\begin{proof}
    For the identity we have $[1,1][n,h] = [\phi(1,n)g(1,h), h] = [n,h]$ and $[n,h][1,1] = [n\phi(h,1)g(h,1),h] = [n,h]$.
    
    Thus, it remains to show that the multiplication is associative. First we consider
    \begin{align*}
        \big([n_1,h_1][n_2,h_2]\big)[n_3,h_3]   &= [n_1\phi(h_1,n_2)g(h_1,h_2),h_1h_2][n_3,h_3] \\
                                        &= [n_1\phi(h_1,n_2)g(h_1,h_2)\phi(h_1h_2,n_3)g(h_1h_2,h_3),h_1h_2h_3] \\
                                        &= n_1\phi(h_1,n_2)\phi(h_1h_2,n_3) \ast [g(h_1,h_2)g(h_1h_2,h_3),h_1h_2h_3].
    \end{align*}
    
    Compare this to
    \begin{align*}
        [n_1,h_1]\big([n_2,h_2][n_3,h_3]\big)   &= [n_1,h_1][n_2\phi(h_2,n_3)g(h_2,h_3),h_2h_3] \\
                                        &= [n_1\phi(h_1,n_2\phi(h_2,n_3)g(h_2,h_3))g(h_1,h_2,h_3),h_1h_2h_3] \\
                                        &= [n_1\phi(h_1,n_2)\phi(h_1,\phi(h_2,n_3))\phi(h_1,g(h_2,h_3))g(h_1,h_2h_3),h_1h_2h_3] \\
                                        &= n_1\phi(h_1,n_2) \ast [\phi(h_1h_2,n_3)\phi(h_1,g(h_2,h_3))g(h_1,h_2h_3),h_1h_2h_3] \\
                                        &= n_1\phi(h_1,n_2)\phi(h_1h_2,n_3) \ast [g(h_1,h_2)g(h_1h_2,h_3),h_1h_2h_3],
    \end{align*}
    which gives us our result. 
\end{proof}

\begin{proposition}\label{prop:ext1}
    Let $N$ be an abelian group, $H$ a monoid, $E$ an admissible equivalence relation, $\phi$ a compatible action and $g$ a factor set. Then $\normalext{N}{k}{(N \times H)/E^\phi_g}{e}{H}$ is a cosetal extension, where $k(n) = [n,1]$ and $e([n,h]) = h$.
\end{proposition}

\begin{proof}
    It is apparent that $k$ and $e$ are well defined monoid homomorphisms.  It is also not hard to see that $k$ is the kernel of $e$. Thus, we must just demonstrate that $e$ is the cokernel of $k$ and that the extension is cosetal.
    
    Let $f\colon (N \times H)/E^\phi_g \to M$ be a monoid homomorphism in which $fk = 0$. It is easy to see that that $[n,h] = [n,1][1,h]$ and so we have
    \begin{align*}
        f([n,h])    &= f([n,1][1,h]) \\
                    &= f(k(n))f([1,h]) \\
                    &= f([1,h]).
    \end{align*}
    We have a map $\ell\colon H \to M$ such that $\ell(h) = f([1,h])$. It is clear that $\ell e = f$ and since $e$ is surjective we must just check that $\ell$ is a homomorphism. We have
    \begin{align*}
        \ell(h)\ell(h') &= f([1,h])f([1,h']) \\
                        &= f([1,h][1,h']) \\
                        &= f([g(h,h'),hh']) \\
                        &= f([g(h,h'),1][1,hh']) \\
                        &= f([1,hh']) \\
                        &= \ell(hh'),
    \end{align*}
    which demonstrates that $e$ is the cokernel.
    
    Now we must show that $\normalext{N}{k}{(N \times H)/E^\phi_g}{e}{H}$ is cosetal. This entails demonstrating that for two equivalence classes $[n,h]$ and $[n',h]$ that there exists an $x \in N$ such that $[x,1][n,h] = [n',h]$. Choosing $x = n'n\inv$ suffices. This completes the proof.
\end{proof}

We know how to extract from a cosetal extension the data $(E,[\phi],g)$, where $E$ is an admissible equivalence relation, $\phi$ a compatible action and $g$ a factor set associated to some section $s$ of $e$.

We also know how to take data $(E,[\phi],g)$ of the same type and generate a cosetal extension \[\normalext{N}{k}{(N \times H)/E^\phi_g}{e}{H}.\] We now relate these two processes to one another.

Fixing $E$ and $[\phi]$ we can consider the set of associated factor sets $\mathcal{F^*}(H,N,E,[\phi])$. This has a natural abelian group structure given by pointwise multiplication.

\begin{proposition}
    $\mathcal{F^*}(H,N,E,[\phi])$ is an abelian group where $(g \cdot g')(h,h') = g(h,h') \cdot g(h,h')$.
\end{proposition}

\begin{proof}
    It is clear that the constant $1$ map is a factor set and that this will behave as an identity.
    
    If $g$ and $g'$ are factor sets, then using commutativity and \cref{lem:calc} we can show that \[[(g \cdot g')(x,y)(g \cdot g')(xy,z),xyz] = [\phi(x,(g \cdot g')(y,z))(g \cdot g')(x,yz),xyz].\]
    
    Finally, we claim that if $g$ is a factor set, then the map $g\inv$ with $g\inv(h,h') = g(h,h')\inv$ is a factor set.
    Observe that $[g(x,y)g(x,yz),1][g\inv(x,y)g\inv(x,yz),xyz] = [1,1]$
    and also
    \begin{align*}
        [g(x,y)g(x,yz),1][\phi(x,g\inv(y,z))g\inv(x,yz),xyz]  &= [g(x,y)g(x,yz)\phi(x,g\inv(y,z))g\inv(x,yz),xyz] \\
        &= [\phi(x,g(y,z))g(x,yz)\phi(x,g\inv(y,z))g\inv(x,yz),xyz] \\ &= [1,1].
    \end{align*}
    Since $g(x,y)g(x,yz)$ is invertible, this gives the desired result.
\end{proof}

From \cref{prop:ext1} we have a map $\rho\colon \mathcal{F^*}(H,N,E,[\phi]) \to \SWSExt(H,N,E,[\phi])$. We do not have a canonical map \[\zeta\colon \SWSExt(H,N,E,[\phi]) \to \mathcal{F^*}(H,N,E,[\phi]),\] as in general there are many factor sets associated to each cosetal extension. We thus would like to quotient $\mathcal{F^*}(H,N,E,[\phi])$ so that all such factor sets are equivalent.

In classical group cohomology and in \cite{martins2016baer} this is a matter of defining the subgroup of inner factor sets. The idea is that if factor sets $g$ and $g'$ correspond to different splittings of the same extension, that they differ by an inner factor set.

Here our situation is slightly more complicated. It is possible to have two factor sets $g$ and $g'$ corresponding to the same splitting of a particular extension. So before we turn to inner factor sets, let us resolve this issue first.

\begin{proposition}
    The equivalence relation $F$ on $\mathcal{F^*}(H,N,E,[\phi])$ defined by $g \sim g'$ if and only if
    \[
        (g(h,h'),hh') \sim (g'(h,h'),hh')
    \]
    is a congruence.
\end{proposition}

\begin{proof}
    Suppose $g \sim g'$ and $r \sim r'$ and consider $[g(h,h')r(h,h'),hh']$ and $[g'(h,h')r'(h,h'),hh']$. \Cref{lem:calc} easily demonstrates their equality.
\end{proof}

Intuitively, this is the correct equivalence relation as it gives $kg(h,h')s(hh') = kg'(h,h')s(hh')$ for all splittings $s$.

Now define $\mathcal{F}(H,N,E,[\phi]) = \mathcal{F^*}(H,N,E,[\phi])/F$ where $F$ is the equivalence relation above. We can now consider the generalisation of inner factor sets.

\begin{definition}
     A factor set $g \in \mathcal{F^*}(H,N,E,[\phi])$ is an \emph{inner factor set} if and only if for some identity preserving $t \colon H \to N$ we have $g = \delta t$ where $\delta t(h,h') = \phi(h,t(h'))t(hh')\inv t(h)$.
\end{definition}

First we will show that if $\rho(g) = \rho(g')$, then $g$ and $g'$ differ by an inner factor set.

\begin{proposition}\label{prop:inner1}
    Let $g,g' \in \mathcal{F^*}(H,N,E,[\phi])$ and let $\rho(g) = \rho(g')$. Then there exists an inner factor set $\delta t$ such that $g' \sim_F \delta t \cdot g$.
\end{proposition}

\begin{proof}
    Let $\normalext{N}{k}{(N \times H)/E^\phi_g}{e}{H}$ and $\normalext{N}{k'}{(N \times H)/E^\phi_{g'}}{e'}{H}$ be the associated cosetal extensions and let $s\colon H \to (N\times H)/E^\phi_g$ be such that $s(h) = [1,h]$ and $s'\colon H \to (N\times H)/E^\phi_{g'}$
    be such that $s'(h) = [1,h]$.
    
    Since $\rho(g) = \rho(g')$ there is an isomorphism of extensions $f \colon (N \times H)/E^\phi_g \to (N \times H)/E^\phi_{g'}$. Now observe that we have
    \begin{align*}
        f([n,h])    &= f([n,1][1,h]) \\
                    &= f([n,1])f([1,h]) \\
                    &= [n,1]f([1,h]).
    \end{align*}
    Then let $f^*\colon H \to N$ be a function which preserves identity and for which $f[1,h] = [f^*(h),1]$. Observe then that $f([n,h]) = [f^*(h)n,h]$.
    We can then define $s^* = fs$ and notice that for $t(h) = f^*(h)\inv$ we have that $s'(h) = kt(h)s^*(h)$. It is also not hard to see that $k'g(h,h')s^*(hh') = s^*(h)s^*(h')$.
    
    We must show that $(\delta t \cdot g(h,h'),hh') \sim (g'(h,h'),hh')$.
    We know that $k'g'(h,h')s'(hh') = s'(h)s'(h')$ and so a single calculation remains.
    \begin{align*}
        k'(\delta t \cdot g)(h,h')s'(hh') &= k'\phi(h,t(h'))k't(hh')\inv k't(h)k'g(h,h')s'(hh') \\
        &= k'\phi(h,t(h'))k't(hh')\inv k't(h)k'g(h,h')k't(hh')s^*(hh') \\
        &= k'\phi(h,t(h'))k't(h)k'g(h,h')s^*(hh') \\
        &= k'\phi(h,t(h'))k't(h)s^*(h)s^*(h') \\
        &= k'\phi(h,t(h'))s'(h)s^*(h') \\
        &= s'(h)k't(h')s^*(h') \\
        &= s'(h)s'(h').
    \end{align*}
    This completes the proof.
\end{proof}

In order to show that equivalence classes of inner factor sets are the appropriate subgroup to quotient by, there is one final result to check.

\begin{proposition}
    Let $g \in \mathcal{F^*}(H,N,E,[\phi])$ and let $\delta t$ be an inner factor set. Then $\rho(g) = \rho(\delta t \cdot g)$.
\end{proposition}

\begin{proof}
    Let $\normalext{N}{k}{(N \times H)/E^\phi_g}{e}{H}$ and $\normalext{N}{k'}{(N \times H)/E^\phi_{\delta t \cdot g}}{e'}{H}$ be the associated cosetal extensions and let $s\colon H \to (N\times H)/E^\phi_g$ be such that $s(h) = [1,h]$ and $s'\colon H \to (N\times H)/E^\phi_{g'}$
    be such that $s'(h) = [1,h]$.
    
    Now inspired by the proof of \cref{prop:inner1} we define a function $f \colon (N \times H)/E^\phi_g \to (N \times H)/E^\phi_{\delta t \cdot g}$ such that $f([n,h]) = [t(h)\inv n,h]$. Since $t(h)\inv$ is invertible, it is clear that $f$ is bijective. Furthermore we have $fk(n) = f([n,1]) = [n,1] = k'(n)$ and $e'f([n,h]) = h = e([n,h])$. It is also clear that $f$
    preserves the identity and so all that remains is to show that $f$
    preserves multiplication.
    
    As before we define $s^* = fs$ and we see that $k't(h)s^*(h) = s'(h)$.

    First we look at $f([n,h])f([n',h'])$. Notice that
    \begin{align*}
        f([n,h])f([n',h'])  &= [t(h)\inv n,h][t(h')\inv n',h'] \\
                            &= [n,1][t(h)\inv,h][n',1][t(h')\inv,h'] \\
                            &= k'(n)s^*(h)k'(n')s^*(h').
    \end{align*}
    
    Next we consider $f([n,h][n',h])$. We have the following.
    \begin{align*}
        f([n,h][n',h])
        &= [t(hh')\inv n\phi(h,n')g(h,h'),hh'] \\
        &= k't(hh')\inv k'(n)k'\phi(h,n')k'g(h,h')s'(hh') \\
        &= k't(hh')\inv k'(n)k'\phi(h,n')k'g(h,h')k't(hh')s^*(hh') \\
        &= k'(n)k'\phi(h,n')k'g(h,h')s^*(hh') \\
        &= k'(n)k'\phi(h,n')s^*(h)s^*(h') \\
        &= k'(n)k'\phi(h,n')k't(h)\inv s'(h)s^*(h') \\
        &= k'(n)k't(h)\inv k'\phi(h,n')s'(h)s^*(h') \\
        &= k'(n)k't(h)\inv s'(h)k'(n')s^*(h) \\
        &= k'(n)s^*(h)k'(n')s^*(h').
    \end{align*}
    This completes the proof.
\end{proof}


%
%
Let $\mathcal{IF^*}(H,N,E,[\phi])$ be the subgroup of inner factor sets and then define the subgroup
\[\mathcal{IF}(H,N,E,[\phi]) = \{[\delta t] : \delta t \in \mathcal{IF^*}(H,N,E,[\phi])\}.\] This then allows us to define $\mathcal{H}^2(H,N,E,[\phi]) = \mathcal{F}(H,N,E,[\phi])/\mathcal{IF}(H,N,E,[\phi])$ and the map \[\zeta\colon \SWSExt(H,N,E,[\phi]) \to \mathcal{H}^2(H,N,E,[\phi])\]
in which an isomorphism class of extensions is sent to the equivalence class of factor sets which generate it.

It is clear that $\zeta\rho$ is the identity. We now show that the reverse also holds true.

\begin{proposition}
    Let $\normalext{N}{k}{G}{e}{H}$ be a cosetal extension, $E$ the associated admissible equivalence relation, $\phi$ the compatible action and $g$ the factor set corresponding to a splitting $s$. Then \[\normalext{N}{k'}{(N \times H)/E^\phi_g}{e'}{H}\] is isomorphic to $\normalext{N}{k}{G}{e}{H}$ --- that is, $\rho\zeta$ is the identity.
\end{proposition}

\begin{proof}
    Let $s$ be a section of $\normalext{N}{k}{G}{e}{H}$ and consider the map $f\colon (N \times H)/E^\phi_g \to G$ where $f([n,h]) = k(n)s(h)$. It is clear that this is a bijective map and preserves the identity. Let us show that it preserves the multiplication.
    \begin{align*}
        f([n,h][n',h])  &= f([n\phi(h,n')g(h,h'),hh']) \\
                        &= k(n)k\phi(h,n')kg(h,h')s(hh') \\
                        &= k(n)k\phi(h,n')s(h)s(h') \\
                        &= k(n)s(h)k(n')s(h') \\
                        &= f([n,h])f([n',h']).
    \end{align*}
    Now it only remains to show $fk' = k$ and $ef = e'$. For the first consider $fk(n) = f([n,1]) = k(n)s(1) = k(n)$. For the second $ef([n,h]) = e(k(n)s(h)) = h$.
\end{proof}

Thus, putting this together we obtain our main result.
\begin{theorem}
    The maps $\rho$ and $\zeta$ give an isomorphism between the set $\SWSExt(H,N,E,[\phi])$ and the abelian group $\mathcal{H}(H,N,E,[\phi])$.
\end{theorem}

Naturally, $\SWSExt(H,N,E,[\phi])$ inherits a multiplication through this isomorphism. It is this that we call the \emph{Baer sum}.

%
%
%

In a follow up paper we will explore the interplay between the cohomology groups $\mathcal{H}^2(H,N,E,[\phi])$ and the order structure of weakly Schreier extensions.

Further work could also be done studying cosetal extensions in full generality, without assuming that the kernel is an abelian group.

\subsection*{Acknowledgements}

I would like to thank Andrea Montoli, Nelson Martins-Ferreira and Graham Manuell for the conversations we had on this topic.

\bibliographystyle{abbrv}
\bibliography{bibliography}
\end{document}